\documentclass[11pt]{article}
\usepackage{amsmath,amsfonts,latexsym,graphicx,amssymb,url,amscd,amsmath}

\usepackage{mathrsfs}
\usepackage{mathtools}
\usepackage{eurosym}
\everymath{\displaystyle}
\newenvironment{proof}[1][Proof]{\textbf{#1.} }{\ \rule{0.5em}{0.5em}}
\usepackage[margin=1in]{geometry}
\usepackage{color}

\newtheorem{theorem}{Theorem}[section]

\newtheorem{remark}[theorem]{Remark}
\newtheorem{lemma}[theorem]{Lemma}

\newtheorem{corollary}[theorem]{Corollary}

\numberwithin{equation}{section}

\newcommand{\beq}{\begin{equation}}
\newcommand{\eeq}{\end{equation}}

\definecolor{darkred}{rgb}{.70,.12,.20}

\definecolor{darkgreen}{rgb}{.20,.52,.14}

\usepackage{setspace}
\newcommand{\bdy}{{\bf y}}
\newcommand{\bdx}{{\bf x}}

\newcommand{\bdz}{{\bf z}}
\newcommand{\bdw}{{\bf w}}

\usepackage{authblk}
\title
{Fractional Korn and Hardy-type inequalities\\
 for vector fields in half space}
\author{Tadele Mengesha \thanks{This research is supported by the NSF grants DMS-1506512 and DMS-1615726.}}

\AtEndDocument{\bigskip{\footnotesize%
  \textsc{Department of Mathematics, \\
  University of Tennessee  Knoxville}\\
\textit{E-mail address}: \texttt{mengesha@utk.edu} 
}}
\date{}
\begin{document}
\maketitle
\abstract{We prove a fractional Hardy-type inequality  for vector fields over the half space based on a modified fractional semi-norm.  A priori, the modified semi-norm is not known to be equivalent to the standard fractional semi-norm and in fact gives a smaller norm, in general. As such, the inequality we prove improves the classical fractional Hardy inequality for vector fields. We will use the inequality to establish the equivalence of a space of functions (of interest) defined over the half space with the classical fractional Sobolev spaces, which amounts to proving a fractional version of the classical Korn's  inequality.  
}
\section{Introduction}
In this paper we prove the following version of the fractional Hardy inequality: for any $s\in (0, 1)$, $ p \in [1, \infty)$ with $ps\neq 1$, there exists a constant $\kappa(p,d,s) > 0$ such that  
  \begin{equation}\label{intro:hardy}
  \int_{\mathbb{R}^{d}_{+}}{|{\bf u}(\bdx)|^{p} \over x_{d}^{ps}} d\bdx \leq \kappa(p, d,s)  \int_{\mathbb{R}^{d}_{+} } \int_{\mathbb{R}^{d}_{+}}\frac{ \left|({\bf u}(\bdy) - {\bf u} ({\bdx})) \cdot \frac{(\bdy -\bdx)}{|\bdy - \bdx|}\right|^{p}}{|\bdy-\bdx|^{d + ps}}d\bdy d\bdx, 
  \end{equation}
  for all ${\bf u}\in C^{1}_{c}(\mathbb{R}^{d}_{+};\mathbb{R}^{d})$. 
  
  For scalar functions, fractional Hardy inequalities that use the standard fractional Sobolev norm $|\cdot|_{W^{s,p}}$ together with their best constant are widely available in the literature. Some of these papers, to cite a few, \cite{BD, D, FS, FS2009} have influenced this work greatly. See also the papers \cite{LS,S,Mazya} for similar estimates.  We note that one could directly use those estimates to obtain similar inequalities for vector fields as well. However, the estimate \eqref{intro:hardy} is distinct from what is available in the sense that in general the seminorm that is used in the right hand side of \eqref{intro:hardy} is smaller than the Gagliardo seminorm $|\cdot|_{W^{s,p}}$, making the estimate tighter. 
  
  Our primary interest in \eqref{intro:hardy} is to apply it to establish  a fractional Korn-type inequality \eqref{intro:korn}  $p=2$. Specifically, we will show that when $s \neq 1/2$, 
  there exists a constant $C(d,s) > 0$ such that
 \begin{equation}\label{intro:korn}
 \int_{\mathbb{R}^{d}_{+} } \int_{\mathbb{R}^{d}_{+}}\frac{ \left|{\bf u}(\bdy) - {\bf u} ({\bdx})\right|^{2}}{|\bdy-\bdx|^{d + 2s}}d\bdy d\bdx \leq C(d,s)\, \int_{\mathbb{R}^{d}_{+} } \int_{\mathbb{R}^{d}_{+}}\frac{ \left|({\bf u}(\bdy) - {\bf u} ({\bdx})) \cdot \frac{(\bdy -\bdx)}{|\bdy - \bdx|}\right|^{2}}{|\bdy-\bdx|^{d + 2s}}d\bdy d\bdx,
\end{equation}
for all ${\bf u}\in L^{2}(\mathbb{R}^{d}_{+}, \mathbb{R}^{d})$.  Shortly, we will explain  why we call the estimate a Korn-type inequality.  For $p=2$, knowing the inequality \eqref{intro:korn} ahead would certainly imply \eqref{intro:hardy} via the standard fractional Hardy inequality. However, at this time we do not know a proof of \eqref{intro:korn} without using \eqref{intro:hardy}.   Nor do we know how to prove \eqref{intro:korn} for $p \neq 2$. 

  The motivation for this  work comes from applications where one needs to work on a function space that uses the seminorm on the right hand side of \eqref{intro:hardy}. To be precise,  we introduce the nonlocal function space of vector fields defined by  $ \mathcal{S}_{\rho, p}(\Omega) := \overline{C^{1}_{c}(\Omega;\mathbb{R}^{d})}^{\|\cdot\|_{\mathcal{S}_{\rho,p}}}, $
  where for  ${\bf u}\in L^{p}(\Omega;\mathbb{R}^{d})$
  \[
  \|{\bf u}\|_{\mathcal{S}_{\rho,p}}^{p} =\|{\bf u}\|_{L^{p}(\Omega)}^{p} +  |{\bf u}|^{p}_{\mathcal{S}_{\rho,p}},
  \]
  and the seminorm $|\cdot|_{\mathcal{S}_{\rho,p}}$ is given by
 \[ |{\bf u}|^{p}_{\mathcal{S}_{\rho,p}} := \int_{\Omega } \int_{\Omega } \rho(\bdy - \bdx)\left|\frac{({\bf u}(\bdy) - {\bf u} ({\bdx}))}{|\bdy-\bdx|} \cdot \frac{(\bdy -\bdx)}{|\bdy - \bdx|}\right|^{p}d\bdy d\bdx. \]
In general, $\Omega$ could be a bounded or an unbounded open subset of $\mathbb{R}^{d}$, and $\rho(\bdz)$, called the kernel,  is a nonnegative integrable function over complement of balls centered at the origin  if $\Omega$ is unbounded, and locally integrable if $\Omega$ is bounded.  The space $C^{1}_{c}({\Omega};\mathbb{R}^{d})$ represents the set of $C^{1}$ functions with compact support in ${\Omega}$.  

For $p = 2,$ the function space $ \mathcal{S}_{\rho, 2}(\Omega)$ has been used in applications, namely in nonlocal continuum mechanics (\cite{Silling2000,Silling2010, Silling2007}) where it appears as the energy space corresponding to the peridynamic  strain energy in a small strain linear model.  Some basic structural  properties of $ \mathcal{S}_{\rho, p}(\Omega)$ have already been investigated in \cite{Mengesha-DuElasticity, Mengesha-Du-non,Du-ZhouM2AN}. It is shown that  for any $1\leq p < \infty$, the  space $\{{\bf u}\in L^{p}(\Omega;\mathbb{R}^{d}): |{\bf u}|^{p}_{\mathcal{S}_{\rho,p}} < \infty\}$ is a separable Banach space with norm $\left(\|{\bf u}\|_{L^{p}}^{p} +  |{\bf u}|^{p}_{\mathcal{S}_{\rho,p}}\right)^{1/p}$, reflexive if $1 < p < \infty$,  and is a Hilbert space when $p =2$.   Under the extra assumption on the kernel that   $\rho$ is radial,  the space is known to support a Poincar\'e-Korn type inequality over subsets that contain no nontrivial zeros of the semi-norm $|\cdot |_{\mathcal{S}_{\rho,p}}$. It is also known that $|{\bf u} |_{\mathcal{S}_{\rho,p}} =0$ if and only if ${\bf u}$ is an affine map with skew-symmetric gradient.  These functional analytic properties of the space were used to demonstrate well posednesss of some nonlocal variational problems using the direct method of calculus of  variations, see \cite{Mengesha-Du-non} for more.

  What distinguishes the space from other nonlocal, difference-based function spaces  is that the seminorm $|{\bf u}|_{\mathcal{S}_{\rho, p}}$ utilizes the smaller projected difference quotient
 \[\mathcal{D}({\bf u})(\bdx, \bdy):= \frac{{\bf u}(\bdy) - {\bf u} ({\bdx})}{|\bdy-\bdx|} \cdot \frac{(\bdy -\bdx)}{|\bdy - \bdx|}\] rather than the full difference quotient, making the space $\mathcal{S}_{\rho, p}(\Omega)$ possibly big. 
 Taking the weighted average of this smaller quantity over enough directions, it is plausible to think that the semi-norm generated will be comparable with the one that is associated with  the full difference quotient. {\em However, this remains unclear}. To be precise, the question of finding a condition on $\rho$, and $\Omega$ so that
 \[
 \mathcal{S}_{\rho, p}(\Omega) = \mathcal{W}_{\rho, p}(\Omega):=\left\{{\bf u}\in L^{p}(\Omega; \mathbb{R}^{d}):  \int_{\Omega } \int_{\Omega } \rho(\bdy - \bdx)\frac{|{\bf u}(\bdy) - {\bf u} ({\bdx}) |^{p} }{|\bdy-\bdx|^{p}}d\bdy d\bdx < \infty\right\},
 \] remains unanswered. 
 It is clear that $\mathcal{W}_{\rho, p}(\Omega) \subset  \mathcal{S}_{\rho, p}(\Omega) $, and when $d=1$, the two space coincide by definition. The problem is therefore in establishing the inclusion $\mathcal{S}_{\rho, p}(\Omega)  \subset  \mathcal{W}_{\rho, p}(\Omega)$. We should remark that for ${\bf u}$ smooth, roughly,  whereas (locally)
\[
\left|{{\bf u}(\bdy) - {\bf u}(\bdx)\over |\bdy -\bdx|}\right|^{p}\approx \left|\nabla {\bf u}(\bdx) {\bdy - \bdx\over |\bdy - \bdx|}\right|^{p} + O(|\bdy-\bdx|)
\]
the projected difference quotient
\[
\mathcal{D}({\bf u})(\bdx, \bdy)\approx  \left|\left\langle\text{Sym}(\nabla {\bf u})(\bdx)  {\bdy - \bdx\over |\bdy - \bdx|},{\bdy - \bdx\over |\bdy - \bdx|}\right\rangle\right|^{p} + O(|\bdy-\bdx|)
\]
where $\text{Sym}(\nabla {\bf u})(\bdx) = 1/2(\nabla {\bf u}(\bdx)^{T} + \nabla {\bf u}(\bdx))$ is the symmetric part of the gradient matrix. This intuition suggests that $\mathcal{S}_{\rho,p}(\Omega)$ is the nonlocal analogue  of the space 
\[
\{ {\bf u}\in L^{p}(\Omega;\mathbb{R}^{d}) : \text{Sym}(\nabla {\bf u}) \in L^{p}(\Omega;\mathbb{R}^{d\times d}) \},  
\] 
which in turn is known to coincide with $W^{1, p}(\Omega;\mathbb{R}^{d})$ via the classical Korn's inequality.  See the paper \cite{Mengesha} that demonstrate, using arguments from \cite{BBM}, that  $\mathcal{S}_{\rho_{n},p}(\Omega) \to W^{1, p}(\Omega;\mathbb{R}^{d})$ when a sequence of radial nonincreasing functions $\rho_{n}$ converge to $ \delta_{0}$, the Dirac Delta measure, in the sense of measures.  As such establishing $\mathcal{S}_{\rho, p}(\Omega) = \mathcal{W}_{\rho, p}(\Omega)$ in the affirmative amounts to proving a version of Korn's inequality for nonlocal function spaces. 

Our interest in the full resolution of the open problem stems from the fact that $\mathcal{W}_{\rho, p}(\Omega)$ may have already known embedding and other smoothness properties. This is the case, for example, when $\rho(\xi) = |\xi|^{-d + p(1-s)}$, for $s\in (0, 1)$, in which case, the space $\mathcal{W}_{\rho, p}(\Omega)$ is precisely the fractional Sobolev space $W^{s,p}(\Omega;\mathbb{R}^{d})$.  This in turn has implications in applications, as one can use these properties to prove regularity of solutions to variational problems whose natural energy space is $ \mathcal{S}_{\rho, p}(\Omega)$.

 With the help of the Hardy inequality \eqref{intro:hardy}, we will establish the equality of function spaces $\mathcal{S}_{\rho,p} (\Omega) = \mathcal{W}_{\rho,p}(\Omega)$ for a special case: when $p=2$, $\rho(\xi) = |\xi|^{-d + 2(1-s)}$, for $s\in (0, 1)$, and $\Omega = \mathbb{R}^{d}_{+}$ or $\mathbb{R}^{d}$.  That is, we prove a version of Korn's inequality for fractional Sobolev spaces $W^{s,2}_{0}(\mathbb{R}^{d}_{+}; \mathbb{R}^{d})$, and this is essentially given in \eqref{intro:korn}. Our approach of the proof of fractional Korn's inequality \eqref{intro:korn} is standard. First we extend vector fields in $\mathcal{S}_{p,s}(\mathbb{R}^{d}_{+})$ (see notation below) to be defined in the whole space in such away that the extended functions belong to $\mathcal{S}_{p,s}(\mathbb{R}^{d})$. Notice that this is a nontrivial task as the commonly used reflection across the hyperplane $x_{d} = 0$ would not be preserving the nonlocal strain energy given by the seminorm $|\cdot|_{\mathcal{S}_{p, s}}$. Nor would extending by zero appropriate, since it is not clear how to control the norm of the extended function. Rather we use an extension operator that has been used by J. A. Nitsche in \cite{Nitsche}  in his simple proof of Korn's second inequality.  As we will see shortly, in showing the boundedness of the extension operator with respect to the seminorm $|\cdot|_{\mathcal{S}_{p, s}}$ 
 we need to first establish the fractional Hardy inequality \eqref{intro:hardy}.  Our proof of \eqref{intro:hardy} follows the recipe given in \cite{FS} where  the so called  ``ground state substitution" is applied to present general Hardy inequalities.   Finally, we will use Fourier transform to show that the space $\mathcal{S}_{2,s}(\mathbb{R}^{d})$ is the same as $W^{s,2}(\mathbb{R}^{d}, \mathbb{R}^{d})$. The latter is already known in \cite{Du-ZhouM2AN}, but for clarity and completeness we will present a proof of it. 

\section{Statement of main result}

\subsection{ Notation} We will be using the following notations throughout the paper. When $\rho(\xi) = |\xi|^{-d + 2(1-s)}$ we denote the function space $ \mathcal{S}_{\rho,p}(\Omega)$ by $ \mathcal{S}_{p,s}(\Omega)$. Points in  $\mathbb{R}^{d}$ will be represented by bold face letters such as $\bdx, \bdy,\bdz$,e.t.c.  Sometimes we may also write component-wise as $\bdx = (\bdx', x_{d})$, where $\bdx' = (x_{1}, x_{2}, \cdots, x_{d-1}) \in \mathbb{R}^{d-1}$. We recall that $\mathbb{R}^{d}_{+} = \{\bdx=(\bdx', x_{d})\in \mathbb{R}^{d}: x_{d} > 0 \}$ and we may use $\Gamma$ to denote the hypersurface $x_{d} = 0$. We write vector fields using boldface letters  and when necessary we write them in components as  ${\bf u} (\bdx) = (u_{1}(\bdx), \cdots, u_{d}(\bdx))$, or as ${\bf u}(\bdx) = ({\bf u}'(\bdx), u_{d}(\bdx))$ where $u_{i}$ is a scalar function for all $i=1,\cdots n$ and ${\bf u}'(\bdx) = (u_{1}(\bdx), \cdots u_{d-1}(\bdx))\in \mathbb{R}^{d-1}$.

\subsection{Main results}
The first result is a fractional Korn-type inequality for vector fields whose statement is given below. 
\begin{theorem}[{\bf Fractional Korn's inequality}] \label{main-theorem} For $d\geq 1$, and $s\neq {1\over 2}$, there exists a constant $C = C(s, d)>0$ such that
\[
|{\bf u}|_{W^{s,2}(\mathbb{R}^{d}_{+};\mathbb{R}^{d}) } \leq C |{\bf u}|_{\mathcal{S}_{2, s}(\mathbb{R}^{d}_{+})}, 
\]
for any ${\bf u}\in C_{c}^{1}(\mathbb{R}^{d}_{+}; \mathbb{R}^{d})$.  In particular, by density, $W_{0}^{s,2}(\mathbb{R}^{d}_{+};\mathbb{R}^{d})  = \mathcal{S}_{2, s}(\mathbb{R}^{d}_{+}).$ 
\end{theorem}
Note in the previous theorem that if $s<{1\over 2}$, then $W_{0}^{s,2}(\mathbb{R}^{d}_{+};\mathbb{R}^{d}) = W^{s,2}(\mathbb{R}^{d}_{+};\mathbb{R}^{d})$. As a consequence, in this case we have the equivalence of $W^{s,2}(\mathbb{R}^{d}_{+};\mathbb{R}^{d})$ and  $\mathcal{S}_{2, s}(\mathbb{R}^{d}_{+})$. As we discussed earlier, the proof of the Theorem \ref{main-theorem}  relies on the following extension theorem, which we believe is interesting in its own right. Notice in particular that the statement is valid for any $p \in [1,\infty)$. 
\begin{theorem} [{\bf Extension operator}]\label{thm:extension}
Let $d\geq 1$, $p\in [1,\infty)$ and $0<s<1$ and $ps\neq 1$. There exists an extension operator $$E : \mathcal{S}_{p,s}(\mathbb{R}^{d}_{+}) \to \mathcal{S}_{p,s}(\mathbb{R}^{d}) $$ and a positive constant $C = C(p,d,s)$ such that for any  ${\bf u}\in \mathcal{S}_{p,s}(\mathbb{R}^{d}_{+})$, and ${\bf U} = E{\bf u}$ we have that ${\bf U} = {\bf u}$ a.e. in $\mathbb{R}^{d}_{+}$, \quad  ${\bf U} = E{\bf u} \in \mathcal{S}_{p,s}(\mathbb{R}^{d})$ and
\[
\begin{split}
&|{\bf U}|_{\mathcal{S}_{p, s}(\mathbb{R}^{d})} \leq C |{\bf u}|_{\mathcal{S}_{p,s}(\mathbb{R}^{d}_{+})}. 
\end{split}
\]
\end{theorem}
The boundedness of the extension operator $E$ is in turn possible by the application of a Hardy-type inequality that we state below. 
\begin{theorem}[{\bf A Hardy-type inequality}]\label{hardy-type} Suppose that $d\geq 1$, $p\in [1,\infty)$, $s\in (0,1)$ and $ps\neq 1$. Then there exists a constant $\kappa(p, d, s) > 0$ such that
\[
\int_{\mathbb{R}^{d}_{+}}\frac{|{\bf u}(\bdx)|^{p}}{x_{d}^{ps}} d\bdx
\leq \kappa(p,d, s) |{\bf u}|^{p}_{\mathcal{S}_{p,s}(\mathbb{R}^{d}_{+})},
\]
for all ${\bf u}  \in C^{1}_{c}(\mathbb{R}^{d}_{+}; \mathbb{R}^{d})$. In particular, by density, the inequality holds for all ${\bf u} \in \mathcal{S}_{p,s}(\mathbb{R}^{d}_{+})$. 
\end{theorem}

\begin{remark} We would like to mention that in this work no effort has been made in obtaining the best constant $\kappa$.  Without accounting for the best constant,  the scaler version of Hardy-type inequality found in \cite{BD, D, FS, FS2009, LS,S,Mazya} can be deduced from Theorem \ref{hardy-type}. Indeed, given a scaler function ${ u}  \in C^{1}_{c}(\mathbb{R}^{d}_{+})$, we can apply the estimate in Theorem \ref{hardy-type} on the vector field $u {\bf e}_{1} = (u, 0, \cdots, 0) \in C^{1}_{c}(\mathbb{R}^{d}_{+}; \mathbb{R}^{d})$ to obtain that 
\[
\int_{\mathbb{R}^{d}_{+}}\frac{|{ u}(\bdx)|^{p}}{x_{d}^{ps}} d\bdx \leq \kappa(p, d,s)  \int_{\mathbb{R}^{d}_{+} } \int_{\mathbb{R}^{d}_{+}}\frac{ \left|{u}(\bdy) - {u} ({\bdx}) \right|^{p}}{|\bdy-\bdx|^{d + ps}}  \frac{|y_{1} -x_{1}|^{p}}{|\bdy - \bdx|^{p}}d\bdy d\bdx. 
\]
\end{remark}

\section{Proof of Hardy-type inequality for vector fields}
Our proof of Theorem \ref{hardy-type} follows the argument in \cite{FS} where the  ``ground state substitution"  ${\bf v}(\bdx) = w(\bdx) {\bf u}(\bdx)$ is used. Here, unlike in \cite{FS}, the appropriately chosen function $w$ in our case solves a "weighted" nonlocal equation that is compatible with the semi-norm $|\cdot|_{\mathcal{S}_{p, s}}$. We first establish  this fact in the following lemma. 
\begin{lemma}\label{euler-lagrange} Suppose that $d\geq 1$, $p\in [1,\infty)$, $s\in (0,1)$, and $ps\neq 1 $. 
Let $\alpha = {(ps-1)\over p}$, and $w(\bdx) = x_{d}^{\alpha}$. Then for any ${\bf v} = ({\bf v}', v_{d})\in \mathbb{R}^{d}$, any $\bdx\in \mathbb{R}^{d}_{+}$
\[
\begin{split}
\lim_{\epsilon\to 0} \int_{ \{\bdy\in \mathbb{R}^{d}_{+}: |x_d - y_d|>\epsilon\}} |{\bf v}\cdot (\bdy-\bdx)|^{p}&\frac{(w(\bdy)-w(\bdx)) |w(\bdy)-w(\bdx)|^{p-2}}{|\bdy-\bdx|^{d+ps+p}}d\bdy \\
&= -\sigma(s,p,d) x_{d}^{-ps} w(\bdx)^{p-1} \mathfrak{f}({\bf v}), 
\end{split}
\]
the convergence being locally uniformly in $\mathbb{R}^{d}_{+}$, where the function $\mathfrak{f}({\bf v})$, and  the positive number $\sigma$ are  given by
\begin{equation}\label{def-f-sigma}
\small{\mathfrak{f}({\bf v}) = \int_{\mathbb{R}^{d-1}}{|{\bf v}'\cdot {\bf z}'  + v_{d}|^{p}  \over (|{\bf z}'|^{2} + 1)^{{d + ps + p\over 2}}} d{\bf z}'  \quad \text{and} \quad  \sigma(d,p,s) = \int_{0}^{1} {|t^\alpha - 1|^{p} \over |t - 1|^{ps +1}} dt.}
\end{equation}
\end{lemma}

\begin{proof}
Once we prove that the limit is valid for $\bdx\in \mathbb{R}^{d}_{+}$, the local uniform convergence follows easily.  Let us fix $\bdx\in \mathbb{R}^{d}_{+}$. For each $\epsilon > 0$, we have that 
\[
\begin{split}
& \int_{\{\bdy\in \mathbb{R}^{d}_{+}: |x_d - y_d|>\epsilon\}} |{\bf v}\cdot (\bdy-\bdx)|^{p}\frac{(w(\bdy)-w(\bdx)) |w(\bdy)-w(\bdx)|^{p-2}}{|\bdy-\bdx|^{d+ps+p}}d\bdy\\ &= \int_{|y_{d} - x_{d}| > \epsilon}(y_{d}^\alpha - x_d^\alpha)|y_d^\alpha - x_d^\alpha|^{p-2} \left\{
\int_{\mathbb{R}^{d-1}} {|{\bf v}'\cdot (\bdy'-\bdx') + v_{d} (y_d - x_d)|^{p} \over (|\bdy'-\bdx'|^{2} + (y_d - x_d)^{2})^{{d + ps + p\over 2}}} d\bdy'\right\} dy_d\\
&=\int_{|y_{d} - x_{d}| > \epsilon}{(y_{d}^\alpha - x_d^\alpha)|y_d^\alpha - x_d^\alpha|^{p-2} \over |y_{d} - x_{d}|^{d + ps}}\left\{
\int_{\mathbb{R}^{d-1}} {|{\bf v}'\cdot {(\bdy'-\bdx')/( y_{d} - x_{d})} + v_{d} |^{p} \over \left(\left|{(\bdy'-\bdx')/(y_{d} - x_{d})}\right|^{2} + 1\right)^{{d + ps +p\over 2}}} d\bdy'\right\} dy_d. 
\end{split}
\]
Making the change of variables, $\bdz' = {\bdy'-\bdx'\over y_d - x_d}$ in the inner integral, we obtain that
\[
\int_{\mathbb{R}^{d-1}} {|{\bf v}'\cdot {(\bdy'-\bdx')/(y_{d} - x_{d})} + v_{d} |^{p} \over \left(\left|{(\bdy'-\bdx')/(y_{d} - x_{d})}\right|^{2} + 1\right)^{{d + ps +p\over 2}}} d\bdy'  = \int_{\mathbb{R}^{d-1}} {|{\bf v}'\cdot \bdz' + v_{d} |^{p} \over (|\bdz'|^{2} + 1)^{{d + ps + p\over 2}}} d\bdz' = \mathfrak{f}({\bf v}).
\]
As a consequence we have that
\begin{equation}\label{eqforw1}
\begin{split}
\lim_{\epsilon \to 0} \int_{\{\bdy\in \mathbb{R}^{d}_{+}: |x_d - y_d|>\epsilon\}}& |{\bf v}\cdot (\bdy-\bdx)|^{p}\frac{(w(\bdy)-w(\bdx)) |w(\bdy)-w(\bdx)|^{p-2}}{|\bdy-\bdx|^{d+ps+p}}d\bdy\\
 &= \mathfrak{f}({\bf v}) \lim_{\epsilon\to 0} \int_{\{y_d > 0: |y_{d} - x_{d}| > \epsilon\}}{(y_{d}^\alpha - x_d^\alpha)|y_d^\alpha - x_d^\alpha|^{p-2} \over |y_{d} - x_{d}|^{d + ps}}dy_{d}. 
\end{split}
\end{equation}
Next, we evaluate the limit in the right hand side of the above equation.  
Making the change of variables, $t = {y_{d}\over x_d}$, and taking the limit in $\epsilon$, we obtain that
\begin{equation}\label{eqforw2}
\begin{split}
\lim_{\epsilon \to 0}&\int_{\{y_{d}>0: |y_{d} - x_{d}| > \epsilon\}} {(|y_{d}|^\alpha - x_d^\alpha)||y_d|^\alpha - x_d^\alpha|^{p-2}\over |y_d - x_d|^{ps + 1}} dy_d \\
&={x_{d}^{\alpha(p-1)}\over x_d^{ps}}\left(\int_{0}^{1} {(t^\alpha -1)|t^\alpha - 1|^{p-2} \over |t - 1|^{ps +1}} dt+ \int_{1}^{\infty}{(t^\alpha -1)|t^\alpha - 1|^{p-2} \over |t - 1|^{ps +1}} dt\right). 
\end{split}
\end{equation}
The later integral can be rewritten by changing variables $t = {1\over r}$ in the sense of improper integrals to yield 
\begin{equation}\label{eqforw3}
\int_{1}^{\infty}{(t^\alpha -1)|t^\alpha - 1|^{p-2} \over |t - 1|^{ps +1}} dt = \int_{0}^{1} {(1- r^\alpha) (1- r^\alpha)^{p-2} \over r^{\alpha(p-1) - ps+1} |1-r|^{ps+1}} dr
\end{equation}
Combining equations \eqref{eqforw1}, \eqref{eqforw2} and \eqref{eqforw3}  and after noting that $\alpha(p-1) - ps+1 = -\alpha$, we have that
\[
\begin{split}
\lim_{\epsilon \to 0} \int_{\{\bdy\in \mathbb{R}^{d}_{+}: |x_d - y_d|>\epsilon\}}& |{\bf v}\cdot (\bdy-\bdx)|^{p}\frac{(w(\bdy)-w(\bdx)) |w(\bdy)-w(\bdx)|^{p-2}}{|\bdy-\bdx|^{d+ps+p}}d\bdy\\
 &= -\mathfrak{f}({\bf v}){w(\bdx)^{p-1}\over x_d^{ps}}\left(\int_{0}^{1} {|t^\alpha - 1|^{p} \over |t - 1|^{ps +1}} dt\right)\\
&=  -\sigma(d, p,s){w(\bdx)^{p-1}\over x_d^{ps}}  \mathfrak{f}({\bf v})
\end{split}
\]
where $\sigma$ is as given in \eqref{def-f-sigma}. 
\end{proof}

\begin{corollary}\label{estimate-for-frakf}
Let ${\bf v} = ({\bf v}',  v_{d})\in \mathbb{R}^{d}$, and  the vector $\mathfrak{f}({\bf v})$ is as given in \eqref{def-f-sigma}. Then for any $p\geq1$, we have that 
\[
\mathfrak{f}({\bf v}) \geq {\eta_{1}\over 2} |v_{d}|^{p} + {\eta_{2}\over 2}|{\bf v}'|^{p}\, , 
\]
where the positive constants $\eta_{1}$ and $\eta_{2}$ are given by 
\[\eta_{1} = \int_{\mathbb{R}^{d-1}} { (|\bdz'|^{2} + 1)^{{-(d + ps + p)\over 2}}} d\bdz', \quad \eta_2 =  \int_{\mathbb{R}^{d-1}} {|z_{1}|^{p} (|\bdz'|^{2} + 1)^{{-(d + ps + p)\over 2}}} d\bdz'.   \]
\end{corollary}
\begin{proof}
The proof follows from the simple convexity inequality:
for any $a, b \in \mathbb{R}$, 
\[
\begin{split}
&|a+ b|^{p} \geq |a|^{p} + p|a|^{p-2}ab, \quad \text{if $p \geq 1$}.
\end{split}
\]
Using the above inequality twice by switching the role of $a$ and $b$,  it follows that for any vectors ${\bf v}$ and ${\bf z} = (\bdz', z_{d})$ in $\mathbb{R}^{d}$
\[
|v_{d} + {\bf v}'\cdot \bdz'  |^{p} \geq {1\over 2} \left[|v_{d}|^{p} +p(|v_{d}|^{p-2} + |{\bf v}'\cdot \bdz' |^{p-2}) v_{d}  {\bf v}'\cdot \bdz'  + |{\bf v}'\cdot \bdz' |^{p}\right]. 
\]
Therefore, we have that 
\[
\mathfrak{f}({\bf v})\geq
|v_{d} |^{p} {1\over 2}\int_{\mathbb{R}^{d-1}} {1\over (|\bdz'|^{2} + 1)^{{d + ps + p\over 2}}} d\bdz' + {1\over 2} \int_{\mathbb{R}^{d-1}} {|{\bf v}'\cdot \bdz'|^{p} \over (|\bdz'|^{2} + 1)^{{d + ps + p\over 2}}} d\bdz', 
\]
where we have made use of the fact that
\[
p.v.\int_{\mathbb{R}^{d-1}} {p|{\bf v}'\cdot \bdz'|^{p-2} |{\bf v}'\cdot \bdz'  \over (|\bdz'|^{2} + 1)^{{d + ps + p\over 2}}} d\bdz' = 0 = p.v.\int_{\mathbb{R}^{d-1}} {{\bf v}'\cdot \bdz'  \over (|\bdz'|^{2} + 1)^{{d + ps + p\over 2}}} d\bdz'. 
\]
The proof will be complete once we realize that 
 for any unit vector ${\bf e}' \in \mathbb{R}^{d-1}$,
 $$0 < \eta_{2}=\int_{\mathbb{R}^{d-1}} {|{\bf e}'\cdot \bdz'|^{p} \over (|\bdz'|^{2} + 1)^{{d + ps + p\over 2}}} d\bdz' = \int_{\mathbb{R}^{d-1}} {|z_{1}|^{p} \over (|\bdz'|^{2} + 1)^{{d + ps + p\over 2}}} d\bdz',$$ after a change of variables with an appropriate rotation $\mathcal{R} \in O(d-1)$ whose first column is ${\bf e}'$.
 \end{proof}
 
 We are now ready to give the proof of the fractional Hardy-type inequality. 
 
\begin{proof}[Proof of Theorem \ref{hardy-type}]
Let $\alpha = {(ps-1)\over p}$, and $w(\bdx) = x_{d}^{\alpha}$ be as given in Lemma \ref{euler-lagrange}.
Let ${\bf u}\in C_{c}^{1}(\mathbb{R}^{d}_{+};\mathbb{R}^{d})$ be given. Let us introduce the function 
\[
V_{\epsilon}(\bdx) = w(\bdx)^{1-p} \int_{\{\bdy\in \mathbb{R}^{d}_{+}: |x_d - y_d|>\epsilon\}} |{\bf u}(\bdx)\cdot (\bdy-\bdx)|^{p}\frac{(w(\bdy)-w(\bdx)) |w(\bdy)-w(\bdx)|^{p-2}}{|\bdy-\bdx|^{d+ps+p}}d\bdy. 
\]
As a consequence of Lemma \ref{euler-lagrange}, Corollary \ref{estimate-for-frakf}, and  the fact that $\mathfrak{f}$ is a smooth function, we have that as $\epsilon \to 0$,
\[
V_{\epsilon}(\bdx) \to -\sigma(d,p,s) x_{d}^{-ps}\mathfrak{f}({\bf u}(x))
\]
uniformly in $\bdx$ in the support of ${\bf u}$, and $V_{\epsilon}(\bdx) = 0$ for $\bdx$ outside of the support of ${\bf u}$.  
 Integrating $V_{\epsilon}(\bdx)$ over $\mathbb{R}^{d}_{+}$, and using the symmetry we have that for any $\epsilon >0,$
\small{\[
\begin{split}
&2\int_{\mathbb{R}^{d}_{+}} V_{\epsilon}(\bdx) d\bdx = \\ 
& \int_{\mathbb{R}^{d}_{+}} \int_{\mathbb{R}^{d}_{+}} \left({|{\bf u}(\bdx)\cdot (\bdy-\bdx)|^{p} \over |w(\bdx)|^{p-1}}  - {|{\bf u}(\bdy)\cdot (\bdy-\bdx)|^{p} \over |w(\bdy)|^{p-1}} \right)
\chi_{\{|x_d - y_d|>\epsilon\}}(\bdy)\frac{(w(\bdy)-w(\bdx)) |w(\bdy)-w(\bdx)|^{p-2}}{|\bdy-\bdx|^{d+ps+p}}d\bdy d\bdx.\end{split}
\]}
We rewrite the above as \[
\begin{split}
\int_{\mathbb{R}_{+}^{d}}\int_{\{\bdy\in \mathbb{R}^{d}_{+}: |x_d - y_d|>\epsilon\}} {\Phi[{\bf u}] (\bdx, \bdy)\over |\bdy -\bdx|^{d + ps+p}} d\bdy d\bdx  &-2\int_{\mathbb{R}^{d}_{+}} V_{\epsilon}(\bdx) d\bdx \\
 &= \int_{\mathbb{R}_{+}^{d}}\int_{\{\bdy\in \mathbb{R}^{d}_{+}: |x_d - y_d|>\epsilon\}} {|({\bf u}(\bdx) - {\bf u}(\bdy))\cdot (\bdy - \bdx)|^{p} \over |\bdy-\bdx|^{d + ps + p} }d\bdy d\bdx,
\end{split}
\]
where
\[\small{
\begin{split}
\Phi[{\bf u}] (\bdx, \bdy) &= |({\bf u}(\bdx) - {\bf u}(\bdy))\cdot (\bdy - \bdx)|^{p} \\
&- \left({|{\bf u}(\bdx)\cdot (\bdy-\bdx)|^{p} \over |w(\bdx)|^{p-1}}  - {|{\bf u}(\bdy)\cdot (\bdy-\bdx)|^{p} \over |w(\bdy)|^{p-1}} \right)(w(\bdx)-w(\bdy)) |w(\bdy)-w(\bdx)|^{p-2}.
\end{split}
}
\]
Our next goal is to show that $\Phi[{\bf u}] (\bdx, \bdy)\geq 0$ for all ${\bf x}$ and ${\bf y}$. To that end, let us first simplify the above expression. Define the scalar functions $\pi(\bdx, \bdy) = {{\bf u}(\bdx)\cdot (\bdy-\bdx)\over w(\bdx)}$.  Then we can rewrite 
\[
\small{
\begin{split}
&\Phi[{\bf u}] (\bdx, \bdy) \\
&= |w(\bdx)\pi(\bdx, \bdy) + w(\bdy)\pi(\bdy, \bdx)|^{p}- \left({|\pi(\bdx, \bdy)|^{p} \over |w(\bdx)|^{p-1}}  - {|\pi(\bdy, \bdx)|^{p} \over |w(\bdy)|^{p-1}} \right)(w(\bdx)-w(\bdy)) |w(\bdy)-w(\bdx)|^{p-2}.
\end{split}}
\] 
Fix $\bdx $ and $\bdy$, and we may assume $w(\bdx) \geq w(\bdy)$ (otherwise work with the reverse inequality).  Let $a = {\pi(\bdx, \bdy) \over -\pi(\bdy, \bdx)}  $ and $t = {w(\bdy)\over w(\bdx)}$. Then $a\in \mathbb{R}$, $0\leq t\leq 1$ and after  factoring appropriate terms we have  
\[
\Phi[{\bf u}] (\bdx, \bdy) = |w(\bdx) \pi(\bdy, \bdx)|^{p} \bigg( (a -t)^{p} - (|a|^p-t)(1 - t)^{p-1}\bigg). 
\]
We now  conclude that $\Phi[{\bf u}] (\bdx, \bdy)  \geq 0$ using \cite[Lemma 2.6]{FS}, where the basic inequality 
$$
|a-t|^{p}\geq (1-t)^{p-1}(|a|^{p}-t)
$$ is shown to hold for all $t\in [0, 1]$, $a\in \mathbb{R}$, and $p\geq1$. This implies that for any $\epsilon > 0$
\[
\int_{\mathbb{R}_{+}^{d}}\int_{\{\bdy\in \mathbb{R}^{d}_{+}: |x_d - y_d|>\epsilon\}} {|({\bf u}(\bdx) - {\bf u}(\bdy))\cdot (\bdy - \bdx)|^{p} \over |\bdy-\bdx|^{d + ps + p} }d\bdy d\bdx \geq -2\int_{\mathbb{R}^{d} _{+}} V_{\epsilon}(\bdx) d\bdx.
\]
We now take the limit on both sides of the inequality to obtain that 
\[
\begin{split}
\int_{\mathbb{R}_{+}^{d}}\int_{\mathbb{R}_{+}^{d}} {|({\bf u}(\bdx) - {\bf u}(\bdy))\cdot (\bdy - \bdx)|^{p} \over |\bdy-\bdx|^{d + ps + p} }d\bdy d\bdx &\geq 2\sigma(p,d,s)\int_{\mathbb{R}^{d} _{+}} x_{d}^{-ps} \mathfrak{f}({\bf u}(\bdx)) d\bdx\\
& \geq \sigma(p,d,s) \int_{\mathbb{R}^{d} _{+}} x_{d}^{-ps}\left(\eta_{1} |u_{d}(\bdx)|^{p} + \eta_{2}|{\bf u}'(\bdx)|^{p} d\bdx\right), 
\end{split}
\]
completing the proof. 

\end{proof}
\begin{remark} For $p\geq 2,$ following \cite[Theorem 1.2]{FS, FS2009} one can write the following fractional  Hardy-type inequality with a remainder term for vector fields as well:  For $0<s<1$, $p\geq 2$ such that $ps\neq 1$, we have that 
\[
\begin{split}
\int_{\mathbb{R}^{d}_{+}} \int_{\mathbb{R}^{d}_{+}} &{|({\bf u}(\bdx) - {\bf u}(\bdy)) \cdot (\bdy-\bdx)|^{p} \over |\bdy- \bdx|^{d + ps + p}} d\bdy d\bdx - \kappa(d,p,s) \int_{\mathbb{R}^{d}_{+}} {|{\bf u} (\bdx)|^{p}\over x_{d}^{ps}}d\bdx\\
& \geq c_{p} \int_{\mathbb{R}^{d}_{+}} \int_{\mathbb{R}^{d}_{+}}  {|(x_{d}^{(1-ps)/p}{\bf u}(\bdx) - y_{d}^{(1-ps)/p}{\bf u}(\bdy)) \cdot (\bdy-\bdx)|^{p} \over |\bdy- \bdx|^{d + ps + p}} {1 \over (y_{d} x_{d})^{(1-ps)/2}} d\bdy d\bdx. 
\end{split}
\]
for all ${\bf u}\in C_{c}^{1}(\mathbb{R}^{d}_{+};\mathbb{R}^{d})$ with support in a bounded set. Here $\kappa (d,p,s)$ is as in Theorem \ref{hardy-type} and $c_{p} = \min_{\tau\in (0, 1/2)} ((1-\tau)^{p} - \tau^{p} + p\tau^{p-1}) $ is in $(0, 1]$. Again this follows the argument used in \cite{FS} and using the elaborate inequality $|a- t|^{p}\geq (1-t)^{p-1}(|a|^{p} - t) + c_{p}t^{p/2}|a-1|^{p}$ which holds for $p\geq 2$, $a\in \mathbb{R}$ and $0\leq t \leq 1$ as proved in \cite[Lemma 2.6]{FS}. In \cite{S}, for scalar functions using an equality of the same spirit, a fractional Hardy-Sobolev-Mazya type inequality is proved for $p=2$. 
\end{remark}
\section{Fractional Korn inequality}
\subsection{Extension operator}
In this subsection we prove Theorem \ref{thm:extension}. For that we need the following property of 
$\mathcal{S}_{p,s}(\mathbb{R}^{d}_{+})$, which essentially says that the space $\mathcal{S}_{p,s}(\mathbb{R}^{d}_{+})$ is stable under certain non-uniform scaling. 
\begin{lemma}\label{linear-map-estimate}For a given $\lambda > 0$, the linear map ${F}_{\lambda}: \mathcal{S}_{p,s}(\mathbb{R}^{d}_{+}) \to \mathcal{S}_{p,s}(\mathbb{R}^{d}_{+})$ given by 
\[
{F}_{\lambda}({\bf v}) (\bdx)= \left({{\bf v}'(\bdx', \lambda x_{d})\over \lambda}, v_{d}(\bdx', \lambda x_{d})\right),\quad \text{for a.e. $\bdx  \in \mathbb{R}^{d}_{+}$, and ${\bf v}(\bdx) = ({\bf v}'(\bdx), v_{d}(\bdx) )$}
\]
is bounded  with the estimate  
$
|{F}_{\lambda}({\bf v})|_{\mathcal{S}_{p,s}(\mathbb{R}^{d}_{+})} \leq \lambda^{{d + ps-2\over p}} |{\bf v}|_{\mathcal{S}_{p,s}(\mathbb{R}^{d}_{+})}. 
$
Moreover, if ${\bf u}\in C_{c}^{1}({\mathbb{R}^{d}_{+}};\mathbb{R}^{d})$, then the difference vector field ${F}_{\lambda}({\bf u}) - {\bf u}\in C^{1}_{c}({\mathbb{R}^{d}_{+}};\mathbb{R}^{d})$  and by Hardy's inequality 
\[
\int_{\mathbb{R}^{d}_{+}} {|{F}_{\lambda}({\bf u}) - {\bf u}|^{p}\over x_{d}^{ps}} \leq C |{\bf u}|_{\mathcal{S}_{p,s}(\mathbb{R}^{d}_{+})}^{p}.
\]
\end{lemma}
\begin{proof}
It suffices to prove only the first part and  only for ${\bf v}$ in $C^{1}_{c}(\mathbb{R}^{d}_{+};\mathbb{R}^{d})$ since the general case follows by density. To that end we may rewrite the $|{F}_{\lambda}({\bf v})|_{\mathcal{S}_{p,s}(\mathbb{R}^{d}_{+})}^{p} $  as 
\[\small{
\begin{split}
&\int_{\mathbb{R}^{d}_{+}}\int_{\mathbb{R}^{d}_{+}} {|({F}_{\lambda}({\bf v})(\bdy) - {F}_{\lambda}({\bf v})(\bdx)) \cdot (\bdy - \bdx)|^{p}\over |\bdy-\bdx|^{d + ps+p}}d\bdy d\bdx \\
& = \int_{\mathbb{R}^{d}_{+}}\int_{\mathbb{R}^{d}_{+}} {|1/\lambda({\bf v}'(\bdy',  \lambda y_{d}) - {\bf v}'(\bdx',  \lambda x_{d})) \cdot (\bdy' - \bdx') + (v_{d}(\bdy',  \lambda y_{d}) - v_{d}(\bdx', \lambda x_{d})) (y_{d}-x_{d})|^{p}\over |\bdy-\bdx|^{d + ps+p}}d\bdy d\bdx \\
& = {1\over \lambda^{p}}\int_{\mathbb{R}^{d}_{+}}\int_{\mathbb{R}^{d}_{+}} {|({\bf v}'(\bdy', \lambda y_{d}) - {\bf v}'(\bdx', \lambda x_{d})) \cdot (\bdy' - \bdx') + (v_{d}(\bdy', \lambda y_{d}) - v_{d}(\bdx', \lambda x_{d})) (\lambda y_{d}-\lambda x_{d})|^{p}\over |\bdy-\bdx|^{d + ps+p}}d\bdy d\bdx.
\end{split}}
\]
Making the change of variables $\bdw = (\bdw', w_{d})\to (\bdy',  \lambda y_{d})$, and $\bdz=(\bdz', z_{d})\to (\bdx', \lambda x_{d})$, we have $\|\bdw - \bdz\| \leq \lambda \|\bdy-\bdx\|$, and therefore 
\[
\begin{split}
|{F}_{\lambda}({\bf v})|_{\mathcal{S}_{p,s}(\mathbb{R}^{d}_{+})}^{p}  =&\int_{\mathbb{R}^{d}_{+}}\int_{\mathbb{R}^{d}_{+}} {|({F}_{\lambda}({\bf v})(\bdy) - {F}_{\lambda}({\bf v})(\bdx)) \cdot (\bdy - \bdx)|^{p}\over |\bdy-\bdx|^{d + ps+p}}d\bdy d\bdx 
\leq \lambda^{d + ps-2}  |{\bf v}|_{\mathcal{S}_{p,s}(\mathbb{R}^{d}_{+})}^{p} \,. 
\end{split}
\]
\end{proof}

With the above  preliminary result at hand, we are now ready to prove the extension theorem. 

\begin{proof} [Proof of Theorem \ref{thm:extension}]
It suffices to prove the theorem only for ${\bf u}\in C_{c}^{1}(\mathbb{R}^{d}_{+};\mathbb{R}^{d})$. The general case follows by density. 
For $\bdx\in \mathbb{R}^{d}$, we write $\bdx = (\bdx', x_{d})$. For a given ${\bf u} = (u_{1}, u_{2}, \cdots u_{d}) = ({\bf u}', u_{d})\in C_{c}^{1}(\mathbb{R}^{d}_{+}; \mathbb{R}^{d})$, define the vector field  $${\bf U}(\bdx', x_{d})= (U_{1}, U_{2}, \cdots, U_{d-1}, U_{d})=({\bf U}', U_{d})$$  as follows:
\[
\left\{
\begin{split}
{\bf U}(\bdx', x_{d})&=
\text{${\bf u}(\bdx', x_{d})$ when $ x_{d} \geq 0$ and }\\
U_{i}(\bdx', x_{d}) &= 2u_{i}(\bdx', -x_{d}) - u_{i}(\bdx', -3x_{d}),  \quad x_{d} < 0, \,\,i= 1, 2, \cdots d-1, \\
U_{d}(\bdx', x_{d}) &= -2 u_{d}(\bdx',-x_{d}) + 3u_{d}(\bdx', -3x_{d}),\quad x_{d} <0.
\end{split}
\right.
\]
We note that $U$ is a Lipschitz function with compact support as all but one of the partial derivatives of $U$ are continuous.  The weak derivative ${\partial U_{d}\over\partial{x_{d}}}$ is discontinuous across the hyperplane $x_{d} = 0$. We reiterate that such kind of extension via generalized reflection has been used in \cite{Nitsche} to prove the classical Korn's inequality. 

We claim that to prove the theorem, it suffices to show the following inequality:  For some $C = C(d, p,s)>0$, 
\begin{equation}\label{intermediate-extension}
|{\bf U}|_{\mathcal{S}_{p,s}(\mathbb{R}^{d})} \leq C\left(|{\bf u}|_{\mathcal{S}_{p,s}(\mathbb{R}^{d}_{+})} +   \int_{\mathbb{R}^{d}_{+} } \int_{\mathbb{R}^{d}_{-} } \frac{|(u_{d}(\bdy', -3y_{d})-u_{d}(\bdy', -y_{d}))x_{d}|^{p}}{|\bdy - \bdx|^{d + (s+1)p}}d\bdy d\bdx\right). 
\end{equation}
Once we prove the above inequality, we can estimate the second term in the right hand side of \eqref{intermediate-extension} by the first term as follows. 
First,  
after change of variables, we have that 
\[
\begin{split}
 \int_{\mathbb{R}^{d}_{+} } \int_{\mathbb{R}^{d}_{-} } \frac{|[u_{d}(\bdy', -3y_{d})-u_{d}(\bdy', -y_{d})]x_{d}|^{p}}{|\bdy - \bdx|^{d + (s+1)p}}d\bdy d\bdx 
 &=   \int_{\mathbb{R}^{d}_{+} } \int_{\mathbb{R}^{d}_{+} } \frac{|[u_{d}(\bdy', 3y_{d})-u_{d}(\bdy', y_{d})]x_{d}|^{p}}{((y_{d} + x_{d})^{2}+|\bdy' - \bdx'|^{2})^{\frac{d + (s+1)p}{2}}}d\bdy d\bdx\\
 & =  \int_{\mathbb{R}^{d}_{+} } |u_{d}(\bdy', 3y_{d})-u_{d}(\bdy', y_{d})|^{p} J(\bdy)d\bdy, 
 \end{split}
 \]
where for each $\bdy \in \mathbb{R}^{d}_{+}$, 
 \[
J(\bdy) =  \int_{\mathbb{R}^{d}_{+} } \frac{|x_{d}|^{p}}{((y_{d} + x_{d})^{2}+|\bdy' - \bdx'|^{2})^{\frac{d + (s+1)p}{2}}}d\bdx. 
 \]
 Let us now  explicitly compute $J(\bdy)$. To that end, we compute that 
\[
\begin{split}
J(\bdy) &= \int_{0}^{\infty} 
\frac{|x_{d}|^{p}} {|y_{d} + x_{d}|^{d + (s+1)p}} \int_{\mathbb{R}^{d-1}} \frac{d\bdx' dx_{d}}{\left(1 + \left( \frac{|\bdy'-\bdx'|}{|y_{d} + x_{d}|}  \right)^{2}\right)^{\frac{d + (s +1)p}{2} }}  \\
& =  \int_{0}^{\infty} 
\frac{|x_{d}|^{p}} {|y_{d} +x_{d}|^{d + (s+1)p}} \int_{\mathbb{R}^{d-1}} \frac{|y_{d}+x_{d}|^{d-1}}{\left(1 +  |{\bf z}'|^{2}\right)^{\frac{d + (s +1)p}{2} }} d{\bf z}' dx_{d} \\
& = \gamma_{1}  \int_{0}^{\infty} 
\frac{|x_{d}|^{p}} {|y_{d} +x_{d}|^{(s+1)p + 1}}  dx_{d}, 
\end{split}
\]
where we have applied a change of variables in the second equality and introduced the notation $\gamma_{1} = \int_{\mathbb{R}^{d-1}} \frac{d{\bf z}'}{\left(1 +  |{\bf z}'|^{2}\right)^{\frac{d + (s +1)p}{2} }}  < \infty.$
Next for each $y_{d}> 0$
 the  simple change of variable $w = x_{d}/y_{d}$ yields that 
  \[
 \int_{0}^{\infty} 
\frac{|x_{d}|^{p}} {|y_{d} +x_{d}|^{(s+1)p + 1}}  dx_{d} = \frac{1}{y^{ps}_{d}} \int_{0}^{\infty} 
\frac{w^{p}}{(1 + w)^{p(s+1) + 1}}= \gamma_{2} \frac{1}{y^{ps}_{d}}   
\]
where $\gamma_{2}  = 
\int_{0}^{\infty} 
\frac{w^{p}}{(1 + w)^{p(s+1) + 1}}<\infty$.  We finally take $\gamma = \gamma_{1}\gamma_{2}$, and obtain 
 \[
 J(\bdy)  = \gamma \frac{1}{y^{ps}_{d}}.   \]
Second, we observe that $u_{d}(\bdy', 3y_{d})-u_{d}(\bdy', y_{d})$ is the $d^{th}$ component of the vector field $F_{3}({\bf u}) - {\bf u}$, where $F_{3}$ is the map as in Lemma \ref{linear-map-estimate} corresponding to $\lambda=3$. Applying Lemma \ref{linear-map-estimate} we have that 
\[
\int_{\mathbb{R}^{d}_{+} } \int_{\mathbb{R}^{d}_{-} } \frac{|[u_{d}(\bdy', -3y_{d})-u_{d}(\bdy', -y_{d})]x_{d}|^{p}}{|\bdy - \bdx|^{d + (s+1)p}}d\bdy d\bdx  \leq  \gamma \int_{\mathbb{R}^{d}_{+} } {|F_{3}({\bf u}) - {\bf u}|^{p}\over y_{d}^{ps}}d\bdy \leq c(d, p,s) |{\bf u}|_{\mathcal{S}_{p,s}(\mathbb{R}^{d}_{+})}^{p}. 
\]
What remains is to prove \eqref{intermediate-extension}. We begin by rewriting the expression as
\[
\int_{\mathbb{R}^{d} } \int_{\mathbb{R}^{d} } \frac{|({\bf U}(\bdy) - {\bf U}(\bdx))\cdot \frac{(\bdy - \bdx)}{|\bdy-\bdx|}|^{p}}{|\bdy - \bdx|^{d + sp}}d\bdy d\bdx = I^{+}(U) + I^{-}(U) + 2I^{\pm}(U),
 \]
 where
 \[
 I^{+}(U) = \int_{\mathbb{R}^{d}_{+} } \int_{\mathbb{R}^{d}_{+} } \frac{|({\bf U}(\bdy) - {\bf U}(\bdx))\cdot \frac{(\bdy - \bdx)}{|\bdy-\bdx|}|^{p}}{|\bdy - \bdx|^{d + sp}}d\bdy d\bdx,
 \]
 \[
 I^{-}(U) = \int_{\mathbb{R}^{d}_{-} } \int_{\mathbb{R}^{d}_{-} } \frac{|({\bf U}(\bdy) - {\bf U}(\bdx))\cdot \frac{(\bdy - \bdx)}{|\bdy-\bdx|}|^{p}}{|\bdy - \bdx|^{d + sp}}d\bdy d\bdx, 
 \] and
\[
 I^{\pm}(U) = \int_{\mathbb{R}^{d}_{+} } \int_{\mathbb{R}^{d}_{-} } \frac{|({\bf U}(\bdy) - {\bf U}(\bdx))\cdot \frac{(\bdy - \bdx)}{|\bdy-\bdx|}|^{p}}{|\bdy - \bdx|^{d + sp}}d\bdy d\bdx.
 \]
 We will estimate each of these terms separately.
Clearly, $I^{+}(U) = | {\bf u}|^{p}_{\mathcal{S}_{p,s}(\mathbb{R}^{d}_{+})}$. We  can bound $I^{-}(U)$ by $ | {\bf u}|^{p}_{\mathcal{S}_{p, s}(\mathbb{R}^{d}_{+})}$ as well.
To that end, notice that
\[
 I^{-}(U) = \int_{\mathbb{R}^{d}_{-} } \int_{\mathbb{R}^{d}_{-} } \frac{|({\bf U}'(\bdy) - {\bf U}'(\bdx))\cdot (\bdy' - \bdx') + (U_{d}(\bdy) - U_{d}(\bdx))(y_{d} - x_{d})|^{p}}{|\bdy - \bdx|^{d + (s+1)p}}d\bdy d\bdx. 
\]
Let us use the definition to write that for any $\bdx, \bdy\in \mathbb{R}^{d}_{-}$
\[
\begin{split}
({\bf U}'(\bdy) &- {\bf U}'(\bdx))\cdot (\bdy' - \bdx')\\
& = (2 {\bf u}'(\bdy', -y_{d}) - {\bf u}'(\bdy', -3y_{d})- (2 {\bf u}'(\bdx', -x_{d}) - {\bf u}'(\bdx', -3x_{d})))\cdot (\bdy'-\bdx')\\
&=2({\bf u}'(\bdy', -y_{d}) - {\bf u}'(\bdx', -x_{d}))\cdot(\bdy' - \bdx')- ({\bf u}'(\bdy', -3y_{d}) - {\bf u}'(\bdx', -3x_{d}))\cdot (\bdy'-\bdx').
\end{split}
\]
Similarly, we also have
\[
\begin{split}
(U_{d}(\bdy)& - U_{d}(\bdx))(y_{d} - x_{d})\\
& = (-2u_{d}(\bdy', -y_{d}) + 3 u_{d}(\bdy', -3y_{d})+[2u_{d}(\bdx', -x_{d}) - 3 u_{d}(\bdx', -3x_{d})])(y_{d} - x_{d})\\
&= 2(-u_{d}(\bdy', -y_{d}) + u_{d}(\bdx', -x_{d}))(y_{d}-x_{d}) + 3 (u_{d}(\bdy', -3y_{d})- u_{d}(\bdx', -3x_{d}))(y_{d} - x_{d}). 
\end{split}
\]
Now using the simple inequality $(a + b)^{p}\leq 2^{p-1}(a^{p} + b^{p})$, we obtain that
\small{\[
\begin{split}
 I^{-}(U) &\leq 2^{2p-1} \int_{\mathbb{R}^{d}_{-} } \int_{\mathbb{R}^{d}_{-} } G_{1}[{\bf u}](\bdx, \bdy) d\bdy d\bdx + 2^{2p-1} \int_{\mathbb{R}^{d}_{-} } \int_{\mathbb{R}^{d}_{-} } G_{2}[{\bf u}](\bdx, \bdy) d\bdy d\bdx \\
  &= I^{-}_{1}({\bf u}) + I^{-}_{2}({\bf u}), 
 \end{split}
\] }
where 
\[
\small{
\begin{split}
&G_{1}[{\bf u}](\bdx, \bdy)= \frac{|( ({\bf u}'(\bdy', -y_{d}) - {\bf u}'(\bdx', -x_{d}))\cdot(\bdy' - \bdx) +( -u_{d}(\bdy', -y_{d}) + u_{d}(\bdx', -x_{d}))(y_{d}-x_{d})|^{p}}{|\bdy - \bdx|^{d + (s+1)p}} \\
&\\
&G_{2}[{\bf u}](\bdx, \bdy) = \frac{|( ({\bf u}'(\bdy', -3y_{d}) - {\bf u}'(\bdx', -3x_{d}))\cdot(\bdy' - \bdx) -3(u_{d}(\bdy', -3y_{d}) - u_{d}(\bdx', -3x_{d}))(y_{d}-x_{d})|^{p}}{|\bdy - \bdx|^{d + (s+1)p}}.
 \end{split}} 
\]
To estimate  $I^{-}_{1}({\bf u})$, we make the change of variables  $y_{d}\to -y_{d}$ and $x_{d}\to -x_{d}$ in the last variables to  obtain
$
I^{-}_{1}({\bf u}) = 2^{2p-1}|{\bf u}|_{\mathcal{S}_{p, s}(\mathbb{R}^{d}_{+})}^{p}
$, after noticing that the distance $|\bdy-\bdx|$ remain unchanged in this transformation.
For $I_{2}^{-}({\bf u})$, we make the change of variable $w_{d} \to -3y_{d}$ and $z_{d}\to -3x_{d}$ and notice that
\[
|\bdy' - \bdx'|^{2} + |y_{d}-x_{d}|^{2} \geq \frac{1}{9} (|\bdy' - \bdx'|^{2} + |w_{d}-z_{d}|^{2}).
\]
After a simple calculation similar to the proof of Lemma \ref{linear-map-estimate}, we obtain that for some constant $c(p,d),$   we have $I^{-}_{2}(U) \leq  c(p,d)|{\bf u}|_{\mathcal{S}_{p, s}(\mathbb{R}^{d}_{+})}^{p}$.
Combining the above two estimates, we have a
\[
I^{-}(U) \leq  c(p,d)|{\bf u}|_{\mathcal{S}_{p, s}(\mathbb{R}^{d}_{+})}^{p}.
\]
Finally, we estimate the mixed integral $I^{\pm}(U)$. 
To that end, for $\bdx\in \mathbb{R}^{d}_{+} $ and $\bdy\in \mathbb{R}^{d}_{-}$, using the definition of the extension
we have that
\[
\begin{split}
({\bf U}'(\bdy) &- {\bf U}'(\bdx))\cdot (\bdy' - \bdx') = (2 {\bf u}'(\bdy', -y_{d}) - {\bf u}'(\bdy', -3y_{d})- {\bf u}'(\bdx', x_{d}) )\cdot (\bdy'-\bdx') \,.
\end{split}
\]
We will rewrite the above as
\begin{equation}\label{eqn:uprime}
\begin{split}
&({\bf U}'(\bdy) - {\bf U}'(\bdx))\cdot (\bdy' - \bdx')\\
&=2[{\bf u}'(\bdy', -y_{d}) - {\bf u}'(\bdx', x_{d})]\cdot(\bdy' - \bdx)- [{\bf u}'(\bdy', -3y_{d}) - {\bf u}'(\bdx', x_{d})]\cdot (\bdy'-\bdx') \,.
\end{split}
\end{equation}
Similarly,
\[
\begin{split}
(U_{d}(\bdy)& - U_{d}(\bdx))(y_{d} - x_{d}) = [-2u_{d}(\bdy', -y_{d}) + 3 u_{d}(\bdy', -3y_{d})- u_{d}(\bdx', x_{d})
](y_{d} - x_{d})
\end{split}
\]
Now denoting the expressions $$A:=-2(u_{d}(\bdy', -y_{d})-u_{d}(\bdx', x_{d}) )\,\,\text{and}\, \,B:=(u_{d}(\bdy', -3y_{d})- u_{d}(\bdx', x_{d}))$$ and writing $
y_{d} - x_{d}= (3y_{d} + x_{d}) -2(y_{d} + x_{d})$, we have that
\[
\begin{split}
-2u_{d}(\bdy', -y_{d}) &+ 3 u_{d}(\bdy', -3y_{d})- u_{d}(\bdx', x_{d})
=A+3B
\end{split}
\]
 and
\begin{equation} \label{eqn:ud}
\begin{split}
(U_{d}(\bdy) &- U_{d}(\bdx))(y_{d} - x_{d})\\
& = (A + 3B)(y_{d}-x_{d}) \\
&=(A + 3B)((3y_{d} + x_{d} )-2(y_{d} + x_{d}))\\
&=A(y_{d} + x_{d}) + B(3y_{d} + x_{d}) -4x_{d}(u_{d}(\bdy',-3y_{d}) -u_{d}(\bdy',-y_{d})). 
\end{split}
\end{equation}
It then follows that inserting the expression in \eqref{eqn:uprime} and \eqref{eqn:ud} into the formula for $I^{\pm}(U)$, and applying triangular inequality that 
\[
I^{\pm}(U)\leq C(p)(\mathcal{I}^{\pm}_{1}(U) + \mathcal{I}^{\pm}_{2}(U) +\mathcal{I}^{\pm}_{3}(U) ), 
\]
where
\[
\small{
\mathcal{I}^{\pm}_{1}(U) =\int_{\mathbb{R}^{d}_{+} } \int_{\mathbb{R}^{d}_{-} } \frac{|({\bf u}'(\bdy', -y_{d})-{\bf u}'(\bdx', x_{d}))\cdot (\bdy' - \bdx') -(u_{d}(\bdy', -y_{d})-u_{d}(\bdx', x_{d}))(y_{d} + x_{d})|^{p}}{|\bdy - \bdx|^{d + (s+1)p}}d\bdy d\bdx,}
\]
\[
\small{
\mathcal{I}^{\pm}_{2}(U) =\int_{\mathbb{R}^{d}_{+} } \int_{\mathbb{R}^{d}_{-} } \frac{|({\bf u}'(\bdy', -3y_{d})-{\bf u}'(\bdx', x_{d}))\cdot (\bdy' - \bdx') - (u_{d}(\bdy', -3y_{d})-u_{d}(\bdx', x_{d}))(3y_{d} + x_{d})|^{p}}{|\bdy - \bdx|^{d + (s+1)p}}d\bdy d\bdx,}
\]
and 
\[
\mathcal{I}^{\pm}_{3}(U) =\int_{\mathbb{R}^{d}_{+} } \int_{\mathbb{R}^{d}_{-} } \frac{|[u_{d}(\bdy', -3y_{d})-u_{d}(\bdy', -y_{d})]x_{d}|^{p}}{|\bdy - \bdx|^{d + (s+1)p}}d\bdy d\bdx. 
\]
To estimate $\mathcal{I}^{\pm}_{1}(U)$ and $\mathcal{I}^{\pm}_{2}(U)$, we make the change of variables $y_{d}\to -y_{d}$ and $y_{d}\to -3y_d$ respectively. We note that in both of these transformation, the distance $|\bdy - \bdx|$ cannot exceed a (uniform) constant multiple of the transformed vector.
We summarize that there exists a constant $c(p, d,s) > 0$ such that
\[
I^{\pm}(U) \leq c(p,d)\left(|u|^{p}_{\mathcal{S}_{p,s}(\mathbb{R}^{d}_{+})} + \int_{\mathbb{R}^{d}_{+} } \int_{\mathbb{R}^{d}_{-} } \frac{|(u_{d}(\bdy', -3y_{d})-u_{d}(\bdy', -y_{d}))x_{d}|^{p}}{|\bdy - \bdx|^{d + (s+1)p}}d\bdy d\bdx\right).
\]
That proves \eqref{intermediate-extension} and therefore completes the proof of the theorem. 
\end{proof}

\subsection{Proof of fractional Korn's inequality}
To prove Theorem \ref{main-theorem}, by the extension theorem above, it suffices to prove fractional Korn's inequality for vector fields that are defined on $\mathbb{R}^{d}$. The remaining part of the paper proves this inequality. 
\begin{theorem}[{\bf Korn-type inequality in the whole space}]
For any $0<s<1$, $ \mathcal{S}_{2,s}(\mathbb{R}^{d}) = W^{2,s}(\mathbb{R}^{d};\mathbb{R}^{d})$. Moreover, there exists constants $C = C(s, d)$ such that
\[
C^{-1} |{\bf u}|_{\mathcal{S}_{2,s}(\mathbb{R}^{d})}  \leq |{\bf u}|_{W^{2,s}(\mathbb{R}^{d};\mathbb{R}^{d})}\leq C |{\bf u}|_{\mathcal{S}_{2,s}(\mathbb{R}^{d})}
\]
\end{theorem}
\begin{proof}
Let ${\bf u}\in \mathcal{S}_{2,s}(\mathbb{R}^{d})$. Then we can write
\[
\begin{split}
 |{\bf u}|_{\mathcal{S}_{2,s}(\Omega)}^{2}  &= \int_{\mathbb{R}^{d}}\int_{\mathbb{R}^{d}} \frac{|({\bf u}(\bdy) - {\bf u}(\bdx))\cdot \frac{(\bdy - \bdx)}{|\bdy-\bdx|}|^{2}}{|\bdy - \bdx|^{d + 2s}}d\bdy d\bdx
  = \int_{\mathbb{R}^{d}} {1\over |\bf h|^{d + 2s}} \|\tau_{{\bf h}}{\bf u}\|^{2}_{L^{2}(\mathbb{R}^{d})} d{\bf h}
 \end{split}
\]
where $\tau_{{\bf h}}{\bf u}(\bdx) = ({\bf u}(\bdx + {\bf h}) - {\bf u}(\bdx))\cdot \frac{{\bf h}}{|{\bf h}|}$. Note that the Fourier transform of $\tau_{{\bf h}}{\bf u}(\bdx)$ is given by
\[
\mathcal{F}(\tau_{{\bf h}}{\bf u})(\xi) = (e^{\imath 2\pi \xi\cdot {\bf h}} -1)\mathcal{F}({\bf u})(\xi)\cdot  \frac{{\bf h}}{|{\bf h}|}\,. 
\]
Using Parseval's identity and after a simple calculation we see that
\[
\begin{split}
\|\tau_{{\bf h}}{\bf u}\|_{L^{2}(\mathbb{R}^{d})}^{2} &=2 \int_{\mathbb{R}^{d}} \left|\mathcal{F}({\bf u}) \cdot \frac{{\bf h}}{|{\bf h}|}\right|^{2} (1 - \cos(2\pi \xi\cdot {\bf h})) d\xi. 
\end{split}
\]
We thus have
\begin{equation}\label{s-norm}
\begin{split}
 |{\bf u}|_{\mathcal{S}_{2,s}(\Omega)}^{2} &= 2 \int_{\mathbb{R}^{d}} {1\over |\bf h|^{d + 2s}}\int_{\mathbb{R}^{d}} \left|\mathcal{F}({\bf u}) \cdot \frac{{\bf h}}{|{\bf h}|}\right|^{2} (1 - \cos(2\pi \xi\cdot {\bf h})) d\xi d {\bf h}\\
 &= 2 \int_{\mathbb{R}^{d}} \langle \mathbb{M}(\xi) \mathcal{F}({\bf u})(\xi), \mathcal{F}({\bf u})(\xi) \rangle  d\xi
\end{split}
\end{equation}
where $\mathbb{M}(\xi)$ is the matrix-valued map
\[
\mathbb{M}(\xi) = \int_{\mathbb{R}^{d}} \frac{1 - \cos(2\pi \xi\cdot {\bf h})}{|{\bf h}|^{d+2+2s}} {\bf h}\otimes {\bf h}\, d{\bf h}. 
\]
 Let us make the change of variable  ${\bf z} = (z_{1}, {\bf z}')= 2\pi |\xi| \mathbb{Q}(\xi){\bf h}$, where for each $\xi \in \mathbb{R}^{d}$,  $\mathbb{Q}(\xi)$ is an orthogonal matrix with its first column given by ${\xi \over |\xi|}$. Then $|{\bf z}| = 2\pi |\xi||{\bf h}|$, and
\[
z_{1} = {\bf e}_{1}\cdot {\bf z} = 2\pi|\xi| \mathbb{Q}(\xi){\bf h}\cdot {\bf e}_{1} = 2\pi \xi\cdot {\bf h}.
\]
Moreover, after simplification
\[
\begin{split}
\mathbb{M}(\xi) &= \int_{\mathbb{R}^{d}} \frac{1 - \cos(2\pi \xi\cdot {\bf h})}{|{\bf h}|^{d+2+2s}} {\bf h}\otimes {\bf h}\, d{\bf h}\\
& = |\xi|^{2s} (2\pi)^{2s}\int_{\mathbb{R}^{d}} \frac{1 - \cos(z_1)}{|{\bf z}|^{d+2+2s}} \mathbb{Q}(\xi)^{T}{\bf z}\otimes \mathbb{Q}(\xi)^{T}{\bf z}\, d{\bf z}\\
& = |\xi|^{2s} (2\pi)^{2s} \mathbb{Q}(\xi)^{T} \left(\int_{\mathbb{R}^{d}} \frac{1 - \cos(z_1)}{|{\bf z}|^{d+2+2s}}{\bf z}\otimes  {\bf z} d{\bf z}\right) \mathbb{Q}(\xi). 
\end{split}
\]
Using change of variable using  rotations it is not difficult to see that the matrix in the integration is a constant diagonal matrix given by
\[
\int_{\mathbb{R}^{d}} \frac{1 - \cos(z_1)}{|{\bf z}|^{d+2+2s}}{\bf z}\otimes  {\bf z} d{\bf z} = \text{diag}(l_{1}, l_{2}, l_{2}. \cdots, l_{2})  = (l_{1} - l_{2})({\bf e}_{1}\otimes {\bf e}_{1}) + l_{2}\mathbb{I}_{d}
\]
where $\mathbb{I}_{d}$ is the $d\times d$ identity matrix and
\[
l_{1} = \int_{\mathbb{R}^{d}} \frac{(1 - \cos(z_1))z_{1}^{2}}{|{\bf z}|^{d+2+2s}} d{\bf z},\quad \quad l_{2} = \int_{\mathbb{R}^{d}} \frac{(1 - \cos(z_1))z_{2}^{2}}{|{\bf z}|^{d+2+2s}} d{\bf z}.
\]
After observing that
\[
\mathbb{Q}(\xi)^{T}  {\bf e}_{1}\otimes {\bf e}_{1}\mathbb{Q}(\xi) = \frac{\xi \otimes \xi}{|\xi|^{2}},
\]
we may rewrite $\mathbb{M}(\xi)$ as
\[
\mathbb{M}(\xi) = (2\pi)^{2s} (l _ {1} - l_{2}) |\xi|^{2s}   \frac{\xi \otimes \xi}{|\xi|^{2}} + l_{2}(2\pi)^{2s} |\xi|^{2s}\mathbb{I}_{d}. 
\]
As a consequence, we have that for any ${\bf v}, {\xi }\in \mathbb{R}^{d}$, we have that
\begin{equation}\label{M-estimate}
\langle \mathbb{M}(\xi){\bf v},{\bf v}\rangle \geq (2\pi)^{2s}|\xi|^{2s} \min\{l_{1}, l_{2}\} |{\bf v}|^{2}.
\end{equation}
We note that $l_{1}$ and $l_{2}$ are positive numbers and satisfy the relation
\[
l_{1} + (d-1)l_{2} = \kappa(d, s) = \int_{\mathbb{R}^{d}} \frac{(1 - \cos(z_1))}{|{\bf z}|^{d+2s}} d{\bf z}.
\]
Combining \ref{s-norm} and \ref{M-estimate} we observe that
\[
|{\bf u}|^{2}_{\mathcal{S}_{2,s}} \geq (2\pi)^{2s +1} \min\{l_{1}, l_{2}\} \int_{\mathbb{R}^{d}} |\xi|^{2s}|\mathcal{F}({\bf u}) (\xi)|^{2}d\xi =  (2\pi)^{2s +1} \min\{l_{1}, l_{2}\}|{\bf u}|^{2}_{H^{s}} \,.
\]
\end{proof}

{\bf Acknowledgment:} The author thanks  the anonymous  referee for a careful reading of the manuscript and comments that made the final version more clear.

\end{document}